\documentclass[12pt]{amsart}

\usepackage{amsmath}
\usepackage{amssymb}
\usepackage{amsthm}
\usepackage{amscd}

\usepackage{color}
\usepackage{hyperref}

\usepackage[latin1]{inputenc}

\usepackage{mathpazo}

\usepackage[scaled=.95]{helvet}
\usepackage{courier}

\swapnumbers
\headsep=1cm \numberwithin{equation}{section}
\newtheorem{theorem}{Theorem}[section]

\newtheorem{corollary}[theorem]{Corollary}

\theoremstyle{definition}
\newtheorem{definition}[theorem]{Definition}
\newtheorem{Assumption}[theorem]{Assumptions}

\newtheorem{remark}[theorem]{Remark}
\newtheorem{remark and definition}[theorem]{Remark and Definition}
\newtheorem{remark and notation}[theorem]{Remark and Notation}

\newtheorem{notation}[theorem]{Notation}

\newtheorem{example}[theorem]{Example}

\newtheorem{conjecture}[theorem]{Conjecture}
\newtheorem{question}[theorem]{Question}

\newcommand\Spec{\operatorname{Spec}}

\newcommand\h{\operatorname{ht}}

\newcommand\Ker{\operatorname{\Ker}}

\newcommand\Ass{\operatorname{Ass}}
\newcommand\Min{\operatorname{Min}}

\def\lra{\longrightarrow}

\def\ZZ{{\mathbb Z}}

\def\p{{\mathfrak p}}

\numberwithin{equation}{section}

\begin{document}

\title[A Formula for Symbolic Powers]{A Formula for Symbolic Powers}


\author[Mantero,\, Miranda-Neto, \, Nagel]{Paolo Mantero, \, Cleto B. Miranda-Neto, \, Uwe Nagel}

\address{Department of Mathematical Sciences, University of Arkansas, Fayetteville, AR 72701, USA}
\email{pmantero@uark.edu}

\address{Departamento de Matem\'atica, Universidade Federal da Para\'iba - 58051-900, Jo\~ao Pessoa, PB, Brazil}
\email{cleto@mat.ufpb.br}

\address{Department of Mathematics, University of Kentucky, 715 Patterson Office Tower,
Lexington, KY 40506-0027, USA}
\email{uwe.nagel@uky.edu}

\thanks{P. Mantero was partially supported by Simons Foundation grant  \#962192.\\
\indent 	C. Miranda-Neto was partially supported by CNPq-Brazil grant 301029/2019-9.\\
\indent U. Nagel was partially supported by Simons Foundation grant  \#636513.}


\begin{abstract} Let $S$ be a Cohen-Macaulay ring which is local or standard graded over a field, and let $I$ be an unmixed ideal that is generically a complete intersection. Our goal in this paper is multi-fold. First, we give a multiplicity-based characterization of when an unmixed subideal $J \subseteq I^{(m)}$ equals the $m$-th symbolic power $I^{(m)}$ of $I$. Second, we provide a saturation-type formula to compute $I^{(m)}$ and employ it to deduce a theoretical criterion for when $I^{(m)}=I^m$. Third, we establish an explicit linear bound on the exponent that makes the saturation  formula effective, and use it to obtain lower bounds for the initial degree of $I^{(m)}$. Along the way, we prove a generalized version of a conjecture raised by Eisenbud and Mazur about ${\rm ann}_S(I^{(m)}/I^m)$, and we propose a conjecture connecting the symbolic defect of an ideal to Jacobian ideals.

\end{abstract}

\maketitle

\section{Introduction}

Symbolic powers of ideals have a relevant role in a number of results and open problems in the literature. Part of the reason is their geometric significance, which appears in a celebrated theorem proved by Zariski and Nagata (later generalized by Hochster and Eisenbud \cite{EH}) stating that the $m$-th symbolic power $I^{(m)}$ of a given radical ideal $I$ in a polynomial ring $R$ over a perfect field consists of all the hypersurfaces passing at least $m$ times through each point of the variety $V(I)$.
(Recall that $I^{(m)}=I^mR_W\cap R$, where $W\subset R$ is the multiplicative set of all $R/I$-regular elements.) For further information, including the well-justified relevance of symbolic powers in the literature, we refer to the survey \cite{symb} and references therein, such as, e.g., \cite{ELS}, \cite{H-Hu} and \cite{Ho-Hu}.

In this note we are concerned with the symbolic powers $I^{(m)}$ of an unmixed ideal $I$ that is generically a complete intersection in a Cohen-Macaulay ring $S$ which is either local or standard graded over a field. 

As an illustration of our work, we observe that a main source of difficulties when investigating symbolic powers $I^{(m)}$ of an unmixed ideal $I$ is the lack of explicit formulas to compute $I^{(m)}$, even in the polynomial ring $k[x,y,z]$. So far, the only  general formula in the literature is the following saturation formula
$$I^{(m)}=I^m :J^{\infty},\qquad \text{where}\qquad J:=\bigcap_{\p\in \Ass^*(I)-\Ass(S/I)}\p$$
and $\Ass^*(I):=\bigcup_{m\geq 1}\Ass(S/I^m)$, (e.g. see Eisenbud \cite[Prop~3.13]{Ei}, Herzog, Hibi and Trung\cite[Section~3]{HHT}). Since, in general, the set $\Ass^*(I)$ is not well--understood, the above is a {\em theoretical} formula and researchers (e.g. \cite[Comment after Exercise~1.55]{Gr}) have lamented the need for an explicit ideal $H$, defined only in terms of $I$ (and not depending on its powers $I^m$ or $m$) such that 
$$I^{(m)}=I^m :H^{\infty}.$$
In one of our main results, Theorem \ref{thm2}, we fill this gap and provide an explicit description of an ideal $H$, easily computed from the presentation of $I$, satisfying the above formula.

We now outline the main results of this paper. In what follows, $c$ denotes the height of the ideal $I$. 

\begin{itemize}
	\item (Theorem \ref{thm1}) Let $m\in \ZZ_+$ and let $J\subseteq I^{(m)}$ be an unmixed ideal with the same height $c$. Then, $J=I^{(m)}$ if and only if
	$${\rm e}(S/J)  =  {\rm e}(S/I)\binom{c+m-1}{c}.$$
	
	\medskip
	
	\item (Theorem \ref{thm2}) Let ${\mathbf F}_c(I)$ be the $c$-th Fitting ideal of $I$. Then, for any $m\in \ZZ_+$,
	$$I^{(m)}  = I^m: {\mathbf F}_c(I)^{\infty}.$$
	As a consequence, we obtain the following theoretical characterization for the equality between ordinary and symbolic powers: $I^{(m)}  =  I^m$ if and only if the ideal ${\mathbf F}_c(I)$ contains an $S/I^m$-regular element (Corollary \ref{cor2}).
	
	\medskip
	
	\item (Theorem \ref{m-1}) We prove that, in general, we can replace the saturation with a fairly small power of ${\mathbf F}_c(I)$; more precisely, we have the formula
	$$I^{(m)}  = I^m: {\mathbf F}_c(I)^{m-1}.$$

	\item (Corollary \ref{alpha}) We prove an explicit lower bound for the initial degree of $I^{(m)}$ when $S$ is a positively graded local domain (not necessarily a polynomial ring), and observe that it is sharp in some cases. 
	
	\medskip
	
	\item (Theorem \ref{tEM}) We prove a generalized version of a conjecture of Eisenbud and Mazur in \cite{EiMa} concerning the annihilator of $I^{(m)}/I^m$ for some determinantal ideals $I$.
	
\end{itemize}

Furthermore, we propose the following conjecture over a polynomial ring $S$ with coefficients in a perfect field $k$.   
\begin{itemize}
	\item[] (Conjecture \ref{conject}) Let $I$ be a radical unmixed ideal of $S$. For any given $g\in S$, let ${\bf J}(g)$ denote the ideal generated by the partial derivatives of $g$ (possibly computed via divided powers if ${\rm char}(k)=p>0$).  Given $m\geq 2$, assume $I^{(m)}=(I^m, g_1, \ldots, g_s)$. Then, the $g_i$'s can be taken so that they satisfy 
	$${\rm rad} \left(\sum_{1\leq j\leq s}{\bf J}(g_j)\right)  = I.$$
	Note that one containment holds easily since ${\bf J}(g_j)\subseteq I$ for all $j$, by the classical Zariski--Nagata theorem.
\end{itemize}

This conjecture is supported by numerous examples, most of which have been computed via the computer algebra software {\em Macaulay2 } \cite{M2}.

\section{Preliminaries}

In this section we establish a few conventions and basic facts used throughout the paper. 

By {\it ring} we tacitly mean Noetherian, commutative ring with identity.
The multiplicity of a module has two slightly different definitions depending on the assumptions on the ring.
\begin{itemize}
	\item If $(S, {\bf m})$ is a local ring with residue field $k$, $\dim(S)=d$, and if $M$ is a finitely generated $S$-module, then the (Hilbert-Samuel) {\it multiplicity} of $M$ over $S$ is $${\rm e}({\bf m}, M)=\displaystyle \lim_{t\rightarrow
		\infty} \frac{d!}{t^d} {\rm dim}_{k}\left(\frac{M}{{\bf m}^tM}\right).$$ 
	
	\item If $S$ is a standard graded algebra over a field $k$ with homogeneous maximal ideal ${\bf m}$, the multiplicity is defined in the same way except that the fraction $d!/t^d$ must be replaced with $(d-1)!/t^{d-1}$. 
	
\end{itemize}
If $M=S$ one simply writes ${\rm e}({\bf m}, S)={\rm e}(S)$ for the multiplicity of $S$.

\smallskip

For a finitely generated $S$-module $M$, the $j$-th {\it Fitting ideal} of $M$ over $S$ is defined from any given $S$-free presentation 
$S^q\stackrel{\varphi}\rightarrow S^{\mu}\rightarrow M\rightarrow 0$ as the ideal $${\mathbf F}_j(M)  =  I_{\mu -j}(\varphi)$$ generated by the $(\mu - j)$-minors of $\varphi$,
where $\varphi$ is regarded as a presentation matrix of $M$ over $S$ (in case $j\geq \mu$ we set ${\mathbf F}_{j}(M)=S$). It is well-known that the Fitting ideals are independent of the choice of free presentation. 

\smallskip

As usual, the height of an ideal $L$ is denoted ${\rm ht}(L)$. Recall, for ideals $I, J$ of $S$, that: 
\begin{itemize}
	\item $I$ is {\em unmixed} of height $c$ if $\h(p)=c$ for every $p\in \Ass_S(S/I)$;
	
	\item The {\em unmixed part of $I$}, denoted $I^{\rm unm}$, is the ideal containing $I$ obtained by intersecting the minimal components of $I$ having minimal height (so $I$ is unmixed if and only if $I=I^{\rm unm}$);
	
	\item $I$ is {\it generically a complete intersection} if the localization $I_p$ can be generated by an $S_p$-sequence for every $p\in {\rm Ass}_S(S/I)$;
	
	\item The {\em saturation of $I$ by $J$} is the ideal $$I: J^{\infty}=\bigcup_{t\geq 1} I: J^t.$$
	By Noetherianess, we have an equality $I:J^{\infty} =  I: J^t$ for all $t \gg 0$.
\end{itemize}

Moreover, we use $\mu(M)$ to denote the minimal cardinality of a generating set of a finitely generated module $M$ (in either local or graded setting).

The key object of interest of the present paper is given in the following definition

\begin{definition}\rm Let $S$ be a ring, $I$ a proper ideal of $S$, $m\geq 2$ an integer. The {\it $m$-th symbolic power of $I$} is the ideal 
	$$I^{(m)} = \bigcap_{p\in {\rm Ass}_S(S/I)}I^m_p \cap S.$$ 
	
\end{definition}

\begin{remark}
	Questions about symbolic powers of $I$ may change drastically depending on the presence, or lack of, embedded primes in $\Ass_S(S/I)$.  For instance, if $(S,{\bf m})$ is local and ${\bf m} \in \Ass_S(S/I)$ then the above definition gives $I^{(t)}=I^t$ for every $t\geq 1$.  To avoid these trivializations, the vast majority of problems and papers about symbolic powers require $I$ to be (at least) radical or unmixed. We will make no exception and assume $I$ is unmixed throughout the paper. One should be aware that our results may fail for radical ideals having minimal primes of different heights; see, for instance, Example \ref{mixed}.  
\end{remark}

\section{A numerical characterization of symbolic powers}
While the inclusion $I^m\subseteq I^{(m)}$ always holds, in general equality occurs in relatively few cases (e.g., if $I$ is a complete intersection in a Cohen-Macaulay ring), and if the inclusion is strict it is often an extremely hard task to determine the minimal generators of $I^{(m)}$, even when $S$ is a polynomial ring in 3 variables over a field (in fact, even determining the degrees of some of these minimal generators!).

For special ideals (or classes of ideals) one may know some generators of $I^{(m)}$ 
and, for some reason, expect them to fully generate $I^{(m)}$. In this case one has the following question:

\begin{question}\label{ques}\rm Let $S$ be a Cohen-Macaulay ring, $I$ is an unmixed $S$-ideal with $\h(I)=c$ and let $m\in \ZZ_+$. If $J\subseteq I^{(m)}$ is an unmixed ideal with ${\rm ht}(J)={\rm ht}(I)$, when does the equality $J=I^{(m)}$ hold?
\end{question}

Theorem \ref{thm1} answers Question \ref{ques} under mild assumptions. It provides a numerical criterion depending solely on $m,c$ and the multiplicities of $S/J$ and $S/I$. It can be used to prove that a ``candidate" ideal is indeed equal to $I^{(m)}$.

\begin{theorem}\label{thm1} Let $S$ be a Cohen-Macaulay ring that is either local or standard graded over a field, and let $m\in \ZZ_+$. Let $I$ be an unmixed, generically complete intersection ideal and let $J\subseteq I^{(m)}$ be an ideal with $\h(J)=\h(I)=:c$.  Then the following assertions are equivalent:
	\begin{enumerate}
		\item[{\rm (i)}] $I^{(m)}=J^{\rm unm}$;
		\item[{\rm (ii)}] ${\rm e}(S/J)={\rm e}(S/I)\binom{c+m-1}{c}$.
	\end{enumerate}
\end{theorem}

\begin{proof}
	Since $I$ is unmixed and the ideals $J\subseteq I^{(m)}$ have the same height, the ideals $J^{\rm unm}, I^{(m)}$ are  unmixed  of the same height. The assumptions on $S$ allow us to employ the associativity formula (see e.g. \cite[Theorem 14.7]{Mat}) from which it follows that the equality $J^{\rm unm}=I^{(m)}$ holds if and only if the two ideals have the same multiplicity. Therefore, to prove the theorem it suffices to show that ${\rm e}(S/I^{(m)})={\rm e}(S/I)\binom{c+m-1}{c}$.
	
	From the associativity formula and $\Ass_S(S/I^{(m)})=\Ass_S(S/I)$, we have
	$${\rm e}(S/I^{(m)}) = \sum_{p\in \Ass_S(S/I)}{\rm e}(S/p){\lambda}(S_p/I^m_p),$$
	where $\lambda(-)$ denotes the length of a module. We now compute ${\lambda}(S_p/I^m_p)$. By assumption, $L:=I_p$ is a complete intersection ideal in the local ring $R:=S_p$. Then $${\rm gr}_L(R)  \cong  (R/L)[x_1,\ldots, x_c]$$ where $x_1,\ldots, x_c$ are indeterminates over $R/L$, and hence, for each $t\geq 0$, the $R/L$-module $L^t/L^{t+1}$ is  free of rank $\binom{c+t-1}{c-1}$. Using the  short exact sequence
	$$0 \lra L^{m-1}/L^m \lra R/L^{m} \lra R/L^{m-1} \lra 0$$
	and induction on $m$, we find $${\lambda}(R/L^m)=\sum_{t=0}^{m-1}{\lambda}(L^t/L^{t+1})=\sum_{t=0}^{m-1}{\lambda}(R/L)\binom{c+t-1}{c-1}={\lambda}(R/L)\binom{c+m-1}{c}.$$
	Finally,
	$$\begin{array}{ll}
		{\rm e}(S/I^{(m)}) &=\displaystyle\sum_{p\in \Ass_S(S/I)}{\rm e}(S/p){\lambda}(S_p/I^m_p)\\
		& =\displaystyle\sum_{p\in \Ass_S(S/I)}{\rm e}(S/p){\lambda}(S_p/I_p)\binom{c+m-1}{c}\\
		& =\displaystyle{\rm e}(S/I)\binom{c+m-1}{c}.
	\end{array}$$ \end{proof}

\begin{example} In the standard graded polynomial ring $S=k[x, y, z, w, t]$, with $k$ a perfect field, consider the following perfect radical ideal of height $c=3$, 
	$$I  =  (xz, xw, yw, yt, zt).$$ 
	First, we note that $S/I^2$ is a Cohen-Macaulay ring, ${\rm e}(S/I)  =  5$ and ${\rm e}(S/I^2)  =  5 \cdot 4 = 20$.  
	Hence, Theorem \ref{thm1} yields $$I^{(2)}=I^2.$$
	Second, we analyze the third symbolic power of $I$.
	Since  $\binom{c+m-1}{c}=\binom{5}{3}=10$, Theorem \ref{thm1} gives, once we have found a ``candidate" unmixed ideal $J\subseteq I^{(3)}$, that we just need ${\rm e}(S/J)=5 \cdot10=50$.
	
	We observe that $g:=xyzwt \in I^{(3)}$ (either by looking at the minimal primes of $I$, or via the Zariski--Nagata theorem, which can be applied as $I$ is radical), 
	and $g\notin I^3$ for degree reasons. Thus,  $I^{(3)}\neq I^3$ and $(I^3, g)\subseteq I^{(3)}$. One can verify that $(I^3,g)$ is unmixed (in fact, perfect) and  that ${\rm e}(S/(I^3, g))  =  50.$  
	By Theorem \ref{thm1}, we conclude that $$I^{(3)}  =  (I^3, xyzwt).$$
\end{example}

\section{A saturation formula for symbolic powers}\label{Applications}

We first specify the set-up for this section.

\begin{Assumption}\label{ass}
	Let $S$ be a Cohen-Macaulay ring, $I$ an unmixed $S$-ideal that is generically a complete intersection with $\h(I)=c$. 
\end{Assumption}

For instance, in this setting $I^{(m)}=I^m$ if and only if $\Ass_S(S/I^m)=\Ass_S(S/I)$.

Our main motivation for this section is the following general question, which is wide-open because of the considerable challenges in computing the symbolic powers of ideals (even perfect ideals of height 2 in $k[x,y,z]$!).

\begin{question} 
	Let $S,I$ be as in Assumptions \ref{ass}, and let $m\in \ZZ_+$. How can $I^{(m)}$ be described effectively?
\end{question}

\subsection{A saturation formula} Our central result of this section  is the following formula for computing $I^{(m)}$ in terms of $I^m$ and a suitable ideal associated to $I$ that does not depend on $m$. 

\begin{theorem}\label{thm2} 
	Let $S,I$ be as in Assumptions \ref{ass}. For any $m\in \ZZ_+$, one has $$I^{(m)}  =  I^m: {\mathbf F}_c(I)^{\infty}.$$
\end{theorem}

\begin{proof}
	Let us first recall that 
	$$V({\mathbf F}_c(I))=\{p \in \Spec\,S\,\mid\, \mu(I_p)>c\}.
	$$ Since $\h(I)=c$ and $I$ is generically a complete intersection it follows that $\h({\mathbf F}_c(I))\geq c+1$ and thus $I^{(m)}:{\mathbf F}_c(I)^t=I^{(m)}$ for every $m\geq 1$ and $t\in \ZZ_+$.
	
	Now, if $\Ass_S(S/I^m) = \Ass_S(S/I)$ then  $I^{(m)}=I^m$, and by the above $I^m=I^m: {\mathbf F}_c(I)^{\infty}$, proving the statement in this case.
	
	We may then assume $\Gamma:=\Ass_S(S/I^m) - \Ass_S(S/I)$ is non-empty, and write $\Gamma=\{p_1,\ldots,p_r\}$. We now claim that $\Gamma\subseteq V({\mathbf F}_c(I))$.

	Indeed, since $\Min(I^m)=\Min(I)$, any $p \in \Gamma$ is an embedded prime of $S/I^m$, and so  $S_p/I^m_p$ has an embedded prime. In particular this implies that $I^m_p$ is not an unmixed ideal, and therefore $I_p$ is not a complete intersection ideal, i.e. $\mu(I_p)>c$. By the above observation, it follows that $p\in V({\mathbf F}_c(I))$, which proves the claim. 
	
	The claim yields that $Q:{\mathbf F}_c(I)^{\infty}=S$ for every $p$-primary ideal $Q$ with $p \in \Gamma$. Writing $I^m=I^{(m)}\cap \left(\bigcap_{j=1}^r Q_j\right)$, where each $Q_j$ is a $p_j$-primary ideal, we then obtain $$I^m:{\mathbf F}_c(I)^{\infty} =\left( I^{(m)}:{\mathbf F}_c(I)^{\infty}\right)\cap \left(\bigcap_{j=1}^r Q_j : {\mathbf F}_c(I)^{\infty}\right) = I^{(m)}: {\mathbf F}_c(I)^{\infty}=I^{(m)}.$$ \end{proof}

\subsection{Some examples, and equality between symbolic and ordinary powers} Below we illustrate that, in Theorem \ref{thm2}, neither of the two standing hypotheses on $I$ can be dropped, even in the presence of the other. We let $k$ denote a perfect field. 

\begin{example}\rm Consider the height 3 ideal $I=(x, y, z)^2\cap (y, z, w)$ in the polynomial ring $S=k[x, y, z, w]$. Note that $I$ is unmixed but not generically a complete intersection. Now, clearly $$I^{(m)}=(x, y, z)^{2m}\cap (y, z, w)^m \quad \mbox{for any} \quad m\geq 2.$$
	A calculation shows that $I^2=(x, y, z)^4\cap (y, z, w)^2$, i.e.,  $I^2=I^{(2)}$, but $I^2 \subsetneq I^2\colon {\mathbf F}_3(I)$. Hence $I^{(2)}\neq I^2\colon {\mathbf F}_3(I)^{\infty}$. An analogous behavior occurs for the third symbolic power.  First we notice that $I^3\colon {\mathbf F}_3(I)^{\infty}\neq I^3$ and  $I^3=(x, y, z)^6\cap (y, z, w)^3=I^{(3)}$, and then we get $$I^{(3)}  \neq  I^3\colon {\mathbf F}_3(I)^{\infty}.$$
\end{example}

\begin{example}\label{mixed}\rm Consider the height 2 ideal $I=(xyz, xtu, zwt, ywu)$ in the polynomial ring $S=k[x, y, z, w, t, u]$. Note that $I$ is generically a complete intersection (even radical), but not unmixed because $(x, w)$ and $(x, y, z)$ are minimal prime divisors of $I$ (for the latter, notice that $I\colon (wtu)=(x, y, z)$). Now, according to \cite[Remark 5.4]{Di-Dra}, we have, for any $m \ge 2$, 
	$$I^{(m)}= I^m + gI^{m-2} \quad \text{with} \quad g=xyzwtu.$$ In particular, $I^{(2)}=(I^2, g)$. However, a computation with {\em Macaulay2} \cite{M2} shows that, for example, the monomial $x^2yz^2t$ lies in $I^2\colon {\mathbf F}_2(I)$ but not in $(I^2, g)$. Hence $I^{(2)}\neq I^2 \colon{\mathbf F}_2(I)^{\infty} $. To study the third symbolic power, by the above formula we have $I^{(3)}=I^3+gI$. Suppose by way of contradiction that this ideal equals $I^3\colon{\mathbf F}_2(I)^{\infty}$. Then we would have $I^3\colon{\mathbf F}_2(I) \subseteq I^3+gI$. But this is a contradiction as a calculation shows that the ideal $I^3\colon{\mathbf F}_2(I)$ contains, for instance, the monomial $$x^3y^2z^3t  \notin  I^3+gI.$$
\end{example}

Next we give another consequence of the formula given in Theorem \ref{thm2}. First, we recall that it is an open problem to find sufficient (and possibly also necessary) conditions for the equality $I^{(m)}=I^m$ to occur (see, e.g., \cite[Section 4.2]{symb}). For instance, a celebrated conjecture in the context of Combinatorial Optimization raised by Conforti and Cornu\'ejols about the Max-Flow-Min-Cut property of clutters was translated by Gitler, Valencia and Villarreal into the following equivalent conjecture: for any squarefree monomial ideal $I$, one has $I^{(m)}=I^m$ for all $m\geq 1$ if and only if $I$ has the so-called {\em packing property}; see \cite[Conjecture 1.6]{Corn}, \cite[Corollary 3.14]{GVV}.

In this context, our next application provides a potentially useful characterization of the equality between symbolic and ordinary powers. 

\begin{corollary}\label{cor2} 
	Let $I$ be as in Assumption \ref{ass}, and let $m\in \ZZ_+$. Then, $$I^{(m)}  =  I^m$$ if and only if
	${\mathbf F}_c(I)$ contains an $S/I^m$-regular element.
\end{corollary}
\begin{proof} We have $I^m: {\mathbf F}_c(I)^{\infty}=I^m: {\mathbf F}_c(I)^t$ for $t$ sufficiently large. Now if $g \in {\mathbf F}_c(I)$ is an $S/I^m$-regular element then so is $g^t \in {\mathbf F}_c(I)^t$, and hence $I^m: {\mathbf F}_c(I)^t=I^m$, which by Theorem \ref{thm2} gives $I^{(m)}=I^m$. The converse is clear.
\end{proof}

\begin{example} Let $S$ be either a polynomial ring (standard graded or localized at the ideal of the origin) or a power series ring, in 5 indeterminates $x, y, z, w, t$ over a field of characteristic zero. Let $I$ be the ideal generated by the maximal minors of the $3\times 4$ matrix \[\left(\begin{array}{cccc}
		x   & w & y  & t \\
		y   & t & z & x  \\
		z  & x & w  & y  
	\end{array}\right).\] Then $I$ is a radical unmixed ideal with ${\rm ht}(I)=2$. Now, by computations we find that $$f:=xy-zt  \in  {\mathbf F}_2(I)$$ satisfies $I^m:(f)=I^m$ at least if $m\leq 7$. By Corollary \ref{cor2}, we have $$I^{(m)}  = I^m \quad \mbox{for} \quad 2\leq m \leq 7.$$
	In fact, by \cite[Theorem 1.1]{MFO},  one has $I^{(m)}  =  I^m$  for any $m\geq 1$.
\end{example}

\section{Effective bounds and proof of a conjecture of Eisenbud and Mazur} 

Theorem \ref{thm2} implies that $I^{(m)}=I^m : {\mathbf F}_c(I)^t$ for $t\gg 0$.  So, for both theoretical and computational purposes one may be interested in determining values of $t$ which can be used. Our main question in this direction is the following.

\begin{question}\label{power}
	Let $S,I$ be as in Assumptions \ref{ass}, and let $m\in \ZZ_+$. For which values of $t_0$ do we have
	$$
	I^{(m)} = I^m : {\mathbf F}_c(I)^t \quad \text{ for all } \quad t\geq t_0?$$ Equivalently, for which values of $t_0$ is $I^{(m)} = I^m : {\mathbf F}_c(I)^{t_0}$?
\end{question}

\begin{remark}\label{inclusion}
	Let $S,I$ be as in Assumptions \ref{ass}. 
	\begin{enumerate}
		\item  Since $\h({\mathbf F}_c(I))>\h(I)$, if one has $I^{(m)}\subseteq I^m:{\mathbf F}_c(I)^{t_0}$ for some $t_0$, then equality holds. 		
		Thus, to show that $I^{(m)} = I^m : {\mathbf F}_c(I)^{t_0}$, it suffices to prove 
		$${\mathbf F}_c(I)^{t_0} \subseteq I^m:I^{(m)} = {\rm ann}_S(I^{(m)}/I^m);$$
		
		\item Part (1) provides an additional reason for our interest in Question \ref{power}: an answer to it automatically gives an estimate on the size of the ideal $I^m:I^{(m)}$, because it shows that ${\mathbf F}_c(I)^{t_0}\subseteq {\rm ann}_S(I^{(m)}/I^m).$
	\end{enumerate}
\end{remark}

\begin{remark}
	Additional interest in Question \ref{power} is provided by the observation that a tight bound on $t_0$ can be used to provide lower bounds for the initial degree of $I^{(m)}$. An instance of this application is exhibited in Corollary \ref{alpha}.
\end{remark}

As an illustration of Remark \ref{inclusion}(2) and the interest in ${\rm ann}_S(I^{(m)}/I^m)$, we recall the following conjecture raised by Eisenbud and Mazur \cite[Conjecture, p. 197]{EiMa} (notice, there is a typo: the ideal ${\mathbf F}_1(I)=I$ should clearly be replaced with ${\mathbf F}_2(I)$, which in this case is the maximal ideal ${\bf m}$ generated by the 3 variables).
\begin{conjecture}\label{EM}{\bf (Eisenbud and Mazur)}
	If $I$ is the ideal of 2-minors of a $2\times 3$ matrix of general linear forms in a polynomial ring in 3 variables, then 
	$${\rm ann}(I^{(m)}/I^m)  =  {\mathbf F}_2(I)^{\left\lfloor\frac{m}{2}\right\rfloor}.$$
\end{conjecture}

Conjecture \ref{EM} implies the following optimal answer to Question \ref{power} for the ideal $I$ of the conjecture: $\lfloor \frac{m}{2}\rfloor$ is the smallest value we can take for $t_0$.

Conjecture \ref{EM} is proved in \cite[Section~2]{RS} using tools from birational geometry. Here, we prove a stronger statement which implies Conjecture \ref{EM}. Our result removes the assumption on the number of variables of the ring and it applies, for instance, to ideals of 2-minors of $2\times N$ matrices of general linear forms, see Remark \ref{special}.

\begin{theorem}\label{tEM}
	Let $k$ be a perfect field, $S=k[x_0,\ldots,x_{N-1}]$ a standard graded polynomial ring in $N\geq 3$ variables, and ${\bf m}=(x_0,\ldots,x_{N-1})$. Consider a radical homogeneous ideal $I$ of $S$ with ${\rm ht}(I)=N-1$ and minimally generated by $\binom{N}{2}$ quadrics. Then, one has 
	$${\rm ann}_S(I^{(m)}/I^m)  =  {\bf m}^{\lfloor\frac{m(N-2)}{N-1}\rfloor}.$$
\end{theorem}

\begin{proof}
	Since $\h(I)=N-1$ and $I$ is a radical ideal, the ring $S/I$ is Cohen--Macaulay. 
	Now, let $\overline{k}$ denote the algebraic closure of $k$, and $T:=S\otimes_k \overline{k} \cong \overline{k}[x_0,\ldots,x_{N-1}]$. Since $I$ is radical and $k$ is perfect, then $IT$ is radical in $T$. Then, by the purity of faithful flat extensions, it suffices to prove the statement in $T$, i.e. we may assume $k$ is algebraically closed. 
	
	Since $S/I$ is a reduced Cohen-Macaulay ring of dimension one and $k$ is algebraically closed, the ideal $I=I_X$ is the defining ideal of a set $X$ of points in $\mathbb P^{N-1}_k$. By assumption, $I$ is generated by ${N \choose 2}$ quadrics. It follows that $|X|=N$ and the points of $X$ span $\mathbb P^{N-1}_k$, therefore, $X$ is a star configuration in $\mathbb P^{N-1}_k$.

	Let us denote by $\alpha(-)$ the initial degree of a homogeneous ideal. One has $\alpha(I_X^{(m)}) = m + \left\lceil\frac{m}{N-1}\right\rceil$ (e.g. \cite[Theorem 4.12(1)]{Ma}) and ${\rm reg}(I_X^{(m)}) = 2m$ (e.g. \cite[Corollary 7.3]{Ma}, or \cite[Corollary 4.4]{Oaxaca}). In particular, this implies $\left[I_X^m\right]_d=\left[I_X^{(m)}\right]_d$ for every $d\geq 2m$. It follows that 
	$${\bf m}^e I_X^{(m)}\subseteq I_X^m \quad \mbox{for} \quad e = 2m-\left(m+\left\lceil\frac{m}{N-1}\right\rceil\right) = \left\lfloor\frac{(N-2)m}{(N-1)} \right\rfloor,$$ and so ${\bf m}^e \subseteq {\rm ann}_S(I_X^{(m)}/I_X^m)$. To prove equality we show that no form of degree $<e$ lies in ${\rm ann}_S(I_X^{(m)}/I_X^m)$. 
	In fact, consider a form $0 \neq g \in I_X^{(m)}$ of minimum degree. If there exists a form $f\in {\rm ann}_S(I_X^{(m)}/I_X^m)$ of degree $<e$,  then $fg\in I_X^m$ and $\deg(fg)<e + \alpha(I_X^{(m)})=2m$, contradicting the fact that $I_X^m$ is generated in degree $2m$.
\end{proof}

\begin{remark}\label{special}
	The assumptions of Theorem \ref{tEM} are satisfied if $I \subset S$ is the ideal generated by the 2-minors of a $2 \times N$ matrix whose entries are general linear forms. So, Conjecture \ref{EM} follows from the special case $N=3$.
\end{remark}

We now return to discussing bounds on the integer $t_0$ of Question \ref{power} in our general setting. We begin by deriving a non-explicit estimate for $t_0$. To this end, recall that the {\it index of nilpotency} of an ideal $L\neq S$ is $${\rm nil}\,L = {\rm min}\{r\geq 1 \, \mid \, (\sqrt{L})^r\subseteq L\}.$$  
Now, with notation as in the proof of Theorem \ref{thm2}, fix $m$ and, for each $j=1,\ldots, r$, let $$\eta_j  =  {\rm max}\{{\rm nil}\,Q_j\}$$
Then an easy adaption of the proof of Theorem \ref{thm2} gives the following estimate.  

\begin{corollary}
	\label{cor:use nilpotency}
	Setting $t_0  = {\rm max}\{\eta_1, \ldots, \eta_r\}$, one has 
	\[
	I^{(m)}  =  I^m\colon {\mathbf F}_c(I)^{t_0}
	\]
\end{corollary}

The difficulty in applying this observation lies is the computation of $\eta_i$.

So, we are led to look for  {\it explicit} bounds for $t_0$. 
For instance, if $m=2$ then it follows by work of Eisenbud and Mazur \cite{EiMa} that the smallest possible value of $t_0$ (i.e., $t_0=1$) works. In fact we have the following theorem, which has evident potential computational applications. 

\begin{theorem}\label{m=2}
	Let $S,I$ be as in Assumptions \ref{ass}. Then one has 
	$$I^{(2)} = I^2 : {\mathbf F}_c(I).$$
\end{theorem}
Thus, $I^{(2)} = I^2 : {\mathbf F}_c(I)^t$ for every $t\geq 1$.

\begin{proof}
	The statement follows by the chain of inclusions
	$$I^{(2)}\subseteq I^2:{\mathbf F}_c(I)\subseteq I^2:{\mathbf F}_c(I)^{\infty}=I^{(2)},$$ where the equality holds by Theorem \ref{thm2}, and the leftmost inclusion follows by \cite[Theorem 6]{EiMa}. 
\end{proof}

Our main result in this section generalizes Theorem \ref{m=2}.  
Its proof is inspired by \cite[Theorem 2.1]{HuRi} and employs linkage. For basic terminology about linkage we refer the reader to \cite{HU} or \cite{Mi}. 

\begin{theorem}\label{m-1}
	Let $S,I$ be as in Assumptions \ref{ass}. Then, for every $m\geq 2$, one has
	$$I^{(m)} = I^m : {\mathbf F}_c(I)^{m-1}.$$
\end{theorem}

\begin{proof}
	
	Without loss of generality we may assume that the residue field is infinite. Thus,  we can take a system of generators of $I$ so that any $c$ of them form a regular sequence. Let $\Delta\in {\mathbf F}_c(I)$ be any ($\mu(I)-c$)-th minor of the presentation matrix of this system of generators. It follows by Cramer's rule that $\Delta$ lies in an ideal geometrically linked to $I$ (see \cite[Proof of Corollary 2.4]{HuRi}). We now prove the following claim.
	
	\smallskip
	
	{\bf Claim.} If $L$ is any ideal geometrically linked to $I$, then
	$LI^{(t)}\subseteq CI^{(t-1)}$ for every $t\geq 2$, where $C=L \cap I$ is the complete intersection defining the link.
	
	\smallskip 
	
	Since $C/C^t$ is a free $S/C^{t-1}$-module, the module $$C/CI^{(t-1)}\cong (C/C^t)/(I^{(t-1)}/C^{t-1})(C/C^t)\cong C/C^t \otimes_{S/C^{t-1}}S/I^{(t-1)}$$ is isomorphic to a direct sum of copies of $S/I^{(t-1)}$ as an $S/C^{t-1}$-module, and thus as an $S$-module. It follows that $\Ass_S(C/CI^{(t-1)}) =  \Ass_S(S/I^{(t-1)}) \subseteq \Ass_S(S/I)$. Now, from the short exact sequence
	$$0 \lra C/CI^{(t-1)} \lra S/CI^{(t-1)} \lra S/C \lra 0$$
	we deduce that $\Ass_S(S/CI^{(t-1)})\subseteq \Ass_S(S/I) \cup \Ass_S(S/C)=\Ass_S(S/C)$. Therefore $CI^{(t-1)}$ is an unmixed ideal of height $c = \h I$. 
	
	We show  $LI^{(t)}\subseteq CI^{(t-1)}$ by proving the inclusion locally at each $p\in \Ass_S(S/CI^{(t-1)})$.  If $I\subseteq p$, then $p\in \Min(I)$. Since $C = L \cap I$ and $\Ass_S (S/L) \cap \Ass_S (S/I)$ is empty, we obtain 
	$I_p = C_p$ and $L_p=S_p$, so
	$(LI^{(t)})_p = C_p^t = (CI^{(t-1)})_p$.
	If $p$ does not contain $I$, then $I_p=S_p$ and $L_p=C_p$, and in this case $(LI^{(t)})_p = C_p = (CI^{(t-1)})_p$. 
	This proves the claim.

	Now, if $\Delta_1,\ldots,\Delta_{m-1}$ are any of the minors generating ${\mathbf F}_c(I)$, then, by the above,  there exist complete intersections $C_1,\ldots,C_{m-1}$ in $I$ such that $\Delta_i I^{(t)}\subseteq C_i I^{(t-1)}$ for every $t\geq 2$. Note we can assume $m\geq 3$ (by Theorem \ref{m=2}). Then,
	$$
	\Delta_1\cdots \Delta_{m-1}I^{(m)}\subseteq \Delta_1\cdots \Delta_{m-2} C_{m-1}I^{(m-1)} \subseteq \ldots \subseteq C_1C_2 \cdots C_{m-1}I \subseteq I^m.
	$$
	Therefore  $I^{(m)}{\mathbf F}_c(I)^{m-1}\subseteq I^m$, and by Remark \ref{inclusion}(1) this concludes the proof.
\end{proof}

\begin{remark}\label{(non)sharp}
	While it is now obvious that for $m=2$ the above result provides an optimal exponent $t_0$ for Question \ref{power}, one may ask how tight  the bound $t_0 \leq m-1$ is in case $m \ge 3$. There are some examples for which the bound is sharp, and examples where actually lower bounds apply. Sharpness can be seen, e.g., in parts (ii) and (iii) of Example \ref{3indet} (to be given in the next section), corresponding respectively to the cases $m=3$ and $m=4$. Finally, to illustrate the opposite situation, just pick $I\subset k[x, y, z]$ as the ideal of $3$ points in $\mathbb P^2$, then ${\bf F}_2(I)={\bf m}$ and by Theorem \ref{tEM} it follows that the optimal $t_0$ is $\lfloor\frac{m}{2}\rfloor$ which is strictly smaller than $m-1$ for $m\geq 3$. 
\end{remark}

We now illustrate a couple of applications of Theorem \ref{m-1}. The first one, Corollary \ref{alpha}, provides a lower bound on the initial degree of $I^{(m)}$ (in a fairly general setting), and the second one, Corollary \ref{corrr}, provides an approximation of $I^{(m)}$ via the annihilator of the $(c+1)$-th exterior power of $I$.

Recall that for a homogeneous ideal $I$ in a graded reduced ring $S$, we clearly have $\alpha(I^{(m)})\leq \alpha(I^m)=m\alpha(I)$. Finding lower bounds on $\alpha(I^{(m)})$ is in general a much harder task. As a first consequence of Theorem \ref{m-1}, we obtain an explicit lower bound for any unmixed, generically complete intersection ideal $I$. We remark that it applies not just to polynomial rings, but actually to any positively graded Cohen--Macaulay domain.

\begin{corollary}\label{alpha}
	In addition to Assumptions \ref{ass}, assume $S$ is a domain positively graded over a field and $I$ is a homogeneous ideal.
	If $I^{(m)} = I^m : {\mathbf F}_c(I)^{t_0} $, then one has $$\alpha(I^{(m)})\geq m\alpha(I) - t_0\alpha({\mathbf F}_c(I)).$$ In particular, $\alpha(I^{(m)})\geq m(\alpha(I) - \alpha({\mathbf F}_c(I))) + \alpha({\mathbf F}_c(I))$. 
\end{corollary}
\begin{proof}
	By assumption ${\mathbf F}_c(I)^{t_0}I^{(m)} \subseteq I^m$, so $\alpha(I^m) \leq \alpha ({\mathbf F}_c(I)^{t_0}I^{(m)})$.
	Since $S$ is a domain,   $\alpha(I^m)=m\alpha(I)$ and $\alpha ({\mathbf F}_c(I)^{t_0}I^{(m)}) =\alpha ({\mathbf F}_c(I)^{t_0}) + \alpha(I^{(m)}) = t_0\alpha ({\mathbf F}_c(I)) + \alpha(I^{(m)})$.
	This proves the first inequality. The second one follows with $t_0 = m-1$ by Theorem \ref{m-1}. 
\end{proof}

If $t_0$ is chosen optimally, one obtains equalities in some cases.
\begin{example}\label{2minors} In $k[x, y, z]$ let $I$ be any one of the following two ideals:
	\begin{itemize}
		\item[$($i$)$] $I=(xy, xz, yz)$ (the defining ideal of the three coordinate points in ${\bf P}^2$), or
		\item[$($ii$)$] $I$ is the ideal of $2$-minors of any $2\times 3$ matrix of general linear forms. 
	\end{itemize}
	In either case by Theorem \ref{tEM} one has $I^{(m)}= I^m :  {\mathbf F}_2(I)^{\left\lfloor\frac{m}{2}\right\rfloor}$. Thus, Corollary \ref{alpha} gives $$\alpha(I^{(m)})\geq m\alpha(I) -\left\lfloor\frac{m}{2}\right\rfloor = 2m - \left\lfloor\frac{m}{2}\right\rfloor =m + \left\lceil\frac{m}{2}\right\rceil. 
	$$
	In the proof of Theorem \ref{tEM} we have already seen that $m + \left\lceil\frac{m}{2}\right\rceil=\alpha(I^{(m)})$.  Hence, in this  example the lower bound in Corollary \ref{alpha} is actually attained.
\end{example}

In general, however, the bound of Corollary \ref{alpha} is not sharp. For example,  in $R=k[x,y,z]$, let $I  =  (x(y^2-z^2), y(z^2-x^2), z(x^2-y^2))$ be the defining ideal of a set of 7 simple points in ${\bf P}^2$ (this is the first of the so-called {\it Fermat ideals}). One has ${\bf F}_2(I)={\bf m}$ and, if $m=2$, of course the smallest $t_0$ one can take is $t_0=m-1=1$; in this case $\alpha (I^{(2)})=6>5=2\cdot 3 - 1$.
\medskip

Our second application of Theorem \ref{m-1} is an ``approximation" result.  It  suggests that the annihilator of the $(c+1)$-th exterior power of $I$, which is known to contain $I$ if ${\rm char}(S)=0$, may be employed to study symbolic powers. We thank M. Mastroeni for pointing out that the rightmost inclusion in the statement below is actually an equality.

\begin{corollary}\label{corrr}  Let $S,I$ be as in Assumption \ref{ass}, and let $m\in \ZZ_+$.
	If ${\rm char}(S)=0$ then
	$$I^m:({\rm ann}\wedge^{c+1}I)^{m-1} \subseteq I^{(m)} = I^m:({\rm ann}\wedge^{c+1}I)^{m-1} {\mathbf F}_{c+1}(I)^{m-1}.$$
\end{corollary}

\begin{proof} Applying \cite[Lemma 2.1]{E-G}, one gets inclusions 
	$$({\rm ann}\wedge^{c+1}I) {\mathbf F}_{c+1}(I)   \subseteq {\mathbf F}_{c}(I)  \subseteq {\rm ann}\wedge^{c+1}I.$$ 
	This yields 
	$$I^m:({\rm ann}\wedge^{c+1}I)^{m-1}   \subseteq I^m:{\mathbf F}_{c}(I)^{m-1}  \subseteq I^m:(({\rm ann}\wedge^{c+1}I) {\mathbf F}_{c+1}(I))^{m-1},
	$$ 
	where the ideal in the middle is $I^{(m)}$ by Theorem \ref{m-1}. Finally, both $F_{c+1}(I)$ and ${\rm ann}\wedge^{c+1}I$ contain $F_c(I)$, so in particular, $\left({\rm ann}\wedge^{c+1}I\right)F_{c+1}(I)$ has height $>c$, thus Remark \ref{inclusion} implies that the right-most inclusion is actually an equality.
\end{proof}

\section{A conjecture connecting symbolic defects and Jacobian ideals}

We conclude the paper with a conjecture supported by a number of computations (performed with the computer algebra system {\em Macaulay2} \cite{M2}) related to the main results of this paper.

For this section, we let $S=k[x_1, \ldots, x_n]$ be a standard graded polynomial ring over a perfect field $k$, with homogeneous maximal ideal ${\bf m}$. For a suitable ideal $I$ of $S$, we propose a conjecture connecting the module $I^{(m)}/I^m$ to Jacobian ideals. In what follows, we denote the radical of $I$ by ${\rm rad}\,I$. We first recall the following concept, introduced in \cite{Gera}.

\begin{definition} Let $I$ be an unmixed or radical ideal $I$ of $S$ and $m\in \ZZ_+$. The {\it $m$-th symbolic defect of $I$} is 
	$${\rm sdefect}(I, m)  =  \mu(I^{(m)}/I^m),$$ i.e., the minimal number of generators of $I^{(m)}/I^m$. 
\end{definition}

Trivially, ${\rm sdefect}(I, 1)=0$. Moreover, it is well-known that if $I$ can be generated by a regular sequence (i.e., a complete intersection ideal) then ${\rm sdefect}(I, m)=0$ for every $m$; see \cite[Appendix 6, Lemma 5]{ZS}.

\begin{notation} Given a polynomial $g\in S$, we recall that the {\it Jacobian} (or {\it gradient}) ideal of $g$ is the ideal  
	$${\bf J}(g)  = \left(\frac{\partial g}{\partial x_1}, \ldots, \frac{\partial g}{\partial x_n}\right).$$  
	(If ${\rm char}(k)=p>0$ one uses divided powers to compute ``partial derivatives".) Since $k$ is perfect, if $I$ is a radical ideal of $S$ then for any $g\in I^{(2)}$ we have  ${\bf J}(g)\subseteq I$ by virtue of the Zariski-Nagata theorem (see \cite[subsection 2.1]{symb}), and hence
	$${\rm rad}\,{\bf J}(g)  \subseteq I.$$
\end{notation}

It is then natural to ask when equality holds. We analyze some examples below, where we also compute symbolic defects.

\begin{example}\label{3indet} (${\rm char}(k)=0$.) Let $I$ be the ideal of $S=k[x, y, z]$ generated by the maximal minors of the $2\times 3$ matrix \[\left(\begin{array}{cccc}
		x+y+z   & x+y & y+z\\
		x   & y & z 
	\end{array}\right).\] Then $I$ is a radical unmixed ideal with ${\rm ht}(I)=2$ and ${\mathbf F}_2(I)={\bf m}$. 
	
	\smallskip
	
	\begin{itemize}
		\item[{\rm (i)}] We have ${\rm sdefect}(I, 2)=1$. Indeed, Theorem \ref{m=2} yields $I^{(2)}  =  I^2\colon {\bf m}$, 
		and one obtains $$I^{(2)}  =  (I^2, g) \quad \text{with } g = x^3-xy^2+y^3-x^2z-3xyz+y^2z+2yz^2+z^3.$$ Here, ${\bf J}(g)=I$. In particular, ${\bf J}(g)$ is radical.

		\medskip

		\item[{\rm (ii)}] We have ${\rm sdefect}(I, 3)=3$. Explicitly, applying Theorem \ref{m-1}, 
		$$I^{(3)}  =  I^3\colon {\mathbf F}_2(I)^{2}  =  I^3\colon {\bf m}^2  =  (I^3, g_1, g_2, g_3), \quad {\rm where}$$
		{\scriptsize
			$$g_1 =  x^5-2x^3y^2+x^2y^3+xy^4-y^5-x^4z-4x^3yz+2x^2y^2z+4xy^3z-$$
			$$2y^4z+3x^2yz^2+3xy^2z^2-3y^3z^2+x^2z^3-3y^2z^3-yz^4,$$
			
			$$g_2  =  x^4y+3x^3y^2-x^2y^3-2xy^4+3y^5-3x^4z-2x^3yz-3x^2y^2z-$$ $$10xy^3z+2y^4z+2x^3z^2+10x^2yz^2+3xy^2z^2+4y^3z^2+x^2z^3-2xyz^3-3xz^4-3yz^4-z^5,$$
			
			$$g_3  =  x^3y^2-xy^4+y^5-x^4z-4xy^3z+y^4z+x^3z^2+3x^2yz^2-xy^2z^2+2y^3z^2-2xyz^3+y^2z^3-xz^4.$$} 
		
		By a computation we find that ${\rm rad}\,({\bf J}(g_1)+{\bf J}(g_2) +{\bf J}(g_3)) = I$.
		
		\medskip

		\item[{\rm (iii)}] We have ${\rm sdefect}(I, 4)=4$. By Theorem \ref{m-1}, 
		$$I^{(4)}  =  I^4\colon {\mathbf F}_2(I)^{3} =  I^4\colon {\bf m}^3  =  (I^4, g_1, g_2, g_3, g_4)$$ 
		for suitable polynomials $g_1, g_2, g_3, g_4\in S$ of large monomial support (which we will not write down explicitly). Their degrees are $6, 7, 7, 7$. Once again it can be confirmed that
		${\rm rad}\,({\bf J}(g_1)+{\bf J}(g_2) +{\bf J}(g_3)+{\bf J}(g_4)) = I$.
	\end{itemize}
\end{example}

\begin{example}\label{Hankel5} (${\rm char}(k)=0$.) Let $I$ be the ideal of $S=k[x, y, z, w, t]$ generated by the $2$-minors of the $3\times 3$ generic Hankel matrix \[\left(\begin{array}{cccc}
		x   & y & z\\
		y  & z & w\\ 
		z  & w & t 
	\end{array}\right).\] Then $I$ is a radical unmixed ideal with ${\rm ht}(I)=3$ and ${\mathbf F}_3(I)={\bf m}^3$. 
	
	\smallskip
	
	\begin{itemize}
		\item[{\rm (i)}] We have ${\rm sdefect}(I, 2)=1$. More precisely, applying Theorem \ref{m-1} one computes 
		$$I^{(2)}  =  I^2\colon {\bf m}^3  =  (I^2, g), \quad g = z^3-2yzw+xw^2+y^2t-xzt.$$ In this case, we have ${\rm rad}\,{\bf J}(g)=I$, and ${\bf J}(g)$ is not radical.

		\medskip

		\item[{\rm (ii)}] We have ${\rm sdefect}(I, 3)=6$. Indeed, using Theorem \ref{m-1}, 
		$$I^{(3)}  = I^3\colon {\mathbf F}_3(I)^{2}  =  I^3\colon {\bf m}^3  =  (I^3, g_1, \ldots, g_6), \quad {\rm where}$$
		{\scriptsize 
			$$g_1  =  y^2z^3-xz^4-2y^3zw+2xyz^2w+xy^2w^2-x^2zw^2+y^4t-2xy^2zt+x^2z^2t,$$ 
			
			$$g_2  =  yz^4-2y^2z^2w-xz^3w+3xyzw^2-x^2w^3+y^3zt-xyz^2t-
			xy^2wt+x^2zwt,$$
			
			$$g_3  =  yz^3w-2y^2zw^2+xyw^3-xz^3t+y^3wt+xyzwt-x^2w^2t-xy^2t^2+x^2zt^2,$$
			
			$$g_4  =  z^4w-2yz^2w^2+xzw^3-yz^3t+3y^2zwt-xz^2wt-xyw^2t-y^3t^2+xyzt^2,$$
			
			$$g_5  =  z^5-4y^2zw^2+xz^2w^2+2xyw^3+y^2z^2t-4xz^3t+2y^3wt+4xyzwt-3x^2w^2t-3xy^2t^2+3x^2zt^2,$$
			
			$$g_6  =  z^3w^2-2yzw^3+xw^4-z^4t+2yz^2wt+y^2w^2t-2xzw^2t-y^2zt^2+xz^2t^2.$$}

		A computation shows that ${\rm rad}\,({\bf J}(g_1)+\ldots +{\bf J}(g_6)) = I$.
	\end{itemize}
\end{example}

These examples are just a few among a much larger set of experiments we computed, for various values of $n$ and $m$. All of them support the following conjecture.  

\begin{conjecture}\label{conject} (${\rm char}(k)=0$.) Let $I$ be a radical unmixed ideal of $S$, and $m\geq 2$. If $s:={\rm sdefect}(I, m)\geq 1$, then we can write $I^{(m)}=(I^m, g_1,\ldots,g_s)$ for some $g_1,\ldots,g_s\in S$ satisfying
	$${\rm rad}\,\left({\bf J}(g_1)+\ldots +{\bf J}(g_s)\right)  =  I.$$
\end{conjecture}

\begin{remark}\label{worth}\rm It is worth observing that:
	\begin{itemize}
		\item[(a)] The inclusion ${\rm rad}\,\left({\bf J}(g_1)+\ldots +{\bf J}(g_s)\right)  \subseteq  I$ holds by the Zariski--Nagata Theorem;
		
		\item[(b)] The $g_i$'s must be chosen carefull -- there are examples of {\em specific} $g_i$'s such that $I^{(m)}=(I^m, g_1,\ldots,g_s)$, but $g_1,\ldots,g_s$ do not satisfy Conjecture \ref{conject}. To illustrate it, we now provide two examples where $S,I$ are graded but the $g_i$'s {\em cannot} be homogeneous -- see Examples \ref{defect1} and \ref{nonhomog} below. Note that these symbolic powers even have symbolic defect 1.
	\end{itemize}
\end{remark}

The first example is a squarefree monomial ideal.
\begin{example}\label{defect1}
	(${\rm char}(k)\neq 2$.) Consider the ideal 
	$$I  = (xy, xz, yz, tu)  \subset  k[x,y,z,t,u].$$
	Then $I$ is a radical ideal having precisely 6 associated primes of height 3. One can check that $I^{(2)}=(I^2, xyz)$, so ${\rm sdefect}(I, 2)=1$. Since $I^2$ is generated in degree 4, $g:=xyz$ is the only homogeneous element $h\in S$ with the property that $I^{(2)}=(I^2,h)$. 
	However, ${\bf J}(g)=(xy,xz,yz)$, and thus we observe that ${\rm rad}\,{\bf J}(g) \neq I$.
	
	On the other hand, we can also write $$I^{(2)}=(I^2, g_1), \quad g_1:=g+t^2u^2=xyz + t^2u^2$$ and we clearly have $${\rm rad}\,{\bf J}(g_1)  = {\rm rad}\,(xy, xz, yz,  t^2u, tu^2)  =  I.$$ So, for the conjecture to be true in this case one needs to take an element $g_1$   which is not  homogeneous.\\
\end{example}

The second example is more geometric in nature. 
\begin{example}\label{nonhomog}
	Let $I$ be the ideal of 5 general points in $\mathbb P^2$, then $I$ is generated by a quadric and two cubics, and $I^{(2)}/I^2$ is generated by a single quintic, so ${\rm sdefect}(I,2)=1$. 
	{\em Macaulay2} \cite{M2} computations suggest that no homogeneous element $g_1$ of degree 5 in $I^{(2)}$ satisfies Conjecture \ref{conject}. 
\end{example}

\end{document}